\documentclass[12pt]{article}
\usepackage{amssymb,amsmath,amsthm,enumerate}

\newcommand{\Z}{\mathbb{Z}}
\newcommand{\cc}{\mathfrak{c}}

\newcommand{\A}{\mathfrak{A}}

\newcommand{\B}{\mathcal{B}}
\newcommand{\OO}{\mathcal{O}}

\newcommand{\Q}{\mathbb{Q}}
\newcommand{\F}{\mathcal{F}}
\renewcommand{\SS}{\mathbb{S}}
\newcommand{\M}{\mathcal{M}}
\newcommand{\T}{\mathcal{T}}
\renewcommand{\a}{\mathfrak{a}}
\renewcommand{\b}{\mathfrak{b}}
\newcommand{\id}{\mathrm{id}}

\newcommand{\ra}{\rightarrow}

\DeclareMathOperator{\ch}{char}

\DeclareMathOperator{\Gal}{Gal}

\newtheorem{theorem}{Theorem}

\newtheorem{prop}[theorem]{Proposition}
\newtheorem{cor}[theorem]{Corollary}

\theoremstyle{definition}
\newtheorem{remark}[theorem]{Remark}
\newtheorem{definition}[theorem]{Definition}

\numberwithin{equation}{section}
\numberwithin{theorem}{section}

\title{Galois scaffolds and semistable extensions}
\author{Kevin Keating \\
Department of Mathematics \\
University of Florida \\
Gainesville, FL 32611 \\
USA \\[.2cm]
{\tt keating@ufl.edu}}

\begin{document}

\maketitle

\begin{abstract}
Let $K$ be a local field and let $L/K$ be a totally
ramified Galois extension of degree $p^n$.  Being
semistable \cite{bon2} and possessing a Galois scaffold
\cite{bce} are two conditions which facilitate the
computation of the additive Galois module structure of
$L/K$.  In this note we show that $L/K$ is semistable if
and only if $L/K$ has a Galois scaffold.  We also give
sufficient conditions in terms of Galois scaffolds for
the extension $L/K$ to be stable.
\end{abstract}

\section{Introduction}

Let $K$ be a local field whose residue field is
perfect with characteristic $p$.  Let $L/K$ be a finite
Galois extension and set $G=\Gal(L/K)$.  Then $L$ is
free of rank 1 over $K[G]$ by the normal basis theorem.
Let $\OO_K$, $\OO_L$ be the integer rings of $K$, $L$,
and let $\M_K$, $\M_L$ be the maximal ideals of these
rings.  For a fractional ideal $\M_L^h$ of $L$ define
the associated order
\[\A_h=\{\gamma\in K[G]:\gamma(\M_L^h)\subset\M_L^h\}.\]
of $\M_L^h$ in $K[G]$.  Suppose $\M_L^h$ is free over
some order $\B$ of $K[G]$.  Then $\B=\A_h$, and $\M_L^h$
is free of rank 1 over $\A_h$.  If $L/K$ is at most
tamely ramified then $\A_h=\OO_K[G]$ and $\M_L^h$ is
free of rank 1 over $\OO_K[G]$ for all $h\in\Z$
\cite[Th.\,1]{ullom}.

     On the other hand, if $L/K$ is wildly ramified then
in most cases $\A_h$ properly contains $\OO_K[G]$, and
hence $\M_L^h$ is not free over $\OO_K[G]$.  The problem
of determining when $\M_L^h$ is free over $\A_h$ for a
totally ramified extension $L/K$ of degree $p^n$ appears
to be quite difficult.  Two classes of extensions for
which a partial answer can be obtained are semistable
extensions \cite{bon2} and extensions with a Galois
scaffold \cite{bce}.  In this note we consider the
relation between semistable extensions and Galois
scaffolds.  In particular, we show that $L/K$ admits
a Galois scaffold if and only if $L/K$ is semistable.

\section{Galois Scaffolds}

In this section we give the definition of a Galois
scaffold.  We also give sufficient conditions for an
extension to admit a Galois scaffold.  This criterion
will be used in section~\ref{semscaf}.

     Let $L/K$ be a totally ramified Galois extension of
degree $p^n$ and set $G=\Gal(L/K)$.  Let
$b_1\le b_2\le\cdots\le b_n$ be the lower ramification
breaks of $L/K$, counted with multiplicity.  Assume that
$p\nmid b_i$ for $1\le i\le n$.  Set
$\SS_{p^n}=\{0,1,\ldots,p^n-1\}$ and write $s\in\SS$ in
base $p$ as
\[s=s_{(0)}p^0+s_{(1)}p^1+\cdots+s_{(n-1)}p^{n-1}.\]
Define a partial order on $\SS_{p^n}$ by $s\preceq t$ if
$s_{(i)}\le t_{(i)}$ for $0\le i\le n-1$.  Then by
Lucas's theorem we have $s\preceq t$ if and only if
$p\nmid\binom{t}{s}$.  Define $\b:\SS_{p^n}\ra\Z$ by
\[\b(s)=s_{(0)}p^0b_n+s_{(1)}p^1b_{n-1}+
\cdots+s_{(n-1)}p^{n-1}b_1.\]
Let $r:\Z\ra\SS_{p^n}$ be the function which maps
$a\in\Z$ onto its least nonnegative residue modulo
$p^n$.  The function $r\circ(-\b):\SS_{p^n}\ra\SS_{p^n}$
is a bijection since $p\nmid b_i$.  Therefore we may
define $\a:\SS_{p^n}\ra\SS_{p^n}$ to be the inverse of
$r\circ(-\b)$.  Let $v_L:L\ra\Z\cup\{\infty\}$ denote
the normalized valuation on $L$.

\begin{definition}[\cite{bce}, Definition 2.6]
\label{scaffold}
Let $\cc\ge1$.  A Galois scaffold
$(\{\Psi_i\},\{\lambda_t\})$ for $L/K$ with precision
$\cc$ consists of elements $\Psi_i\in K[G]$ for
$1\le i\le n$ and $\lambda_t\in L$ for all $t\in\Z$ such
that the following hold:
\begin{enumerate}[(i)]
\item $v_L(\lambda_t)=t$ for all $t\in\Z$.
\item $\lambda_{t_1}\lambda_{t_2}^{-1}\in K$
whenever $t_1\equiv t_2\pmod{p^n}$.
\item $\Psi_i(1)=0$ for $1\le i\le n$.
\item For $1\le i\le n$ and $t\in\Z$ there exists
$u_{it}\in\OO_K^{\times}$ such that the following
congruence modulo $\lambda_{t+p^{n-i}b_i}\M_L^{\cc}$
holds:
\[\Psi_i(\lambda_t)\equiv\begin{cases} 
u_{it}\lambda_{t+p^{n-i}b_i}&
\mbox{if }\a(r(t))_{(n-i)}\ge1, \\ 
0&\mbox{if }\a(r(t))_{(n-i)}=0.
\end{cases} \]
\end{enumerate}
A Galois scaffold for $L/K$ with infinite precision
consists of the above data with the congruence in (iv)
replaced by equality.
\end{definition}

     In \cite[Th.\,3.1]{bce} sufficient conditions are
given for an ideal $\M_L^h$ in an extension $L/K$ with a
Galois scaffold to be free over its associated order
$\A_h$.  When the precision of the scaffold is
sufficiently large these sufficient conditions are shown
to be necessary as well.

     Let $(\{\Psi_i\},\{\lambda_t\})$ be a Galois
scaffold for $L/K$.  For $s\in\SS_{p^n}$ set
\[\Psi^{(s)}=\Psi_n^{s_{(0)}}\Psi_{n-1}^{s_{(1)}}
\ldots\Psi_2^{s_{(n-2)}}\Psi_1^{s_{(n-1)}}.\]
As explained on page 974 of \cite{bce}, it follows from
the definition of a Galois scaffold that for every
$t\in\Z$ there is $U_{st}\in\OO_K^{\times}$ such that
the following holds modulo
$\lambda_{t+\b(s)}\M_L^{\cc}$:
\begin{equation} \label{Psis}
\Psi^{(s)}(\lambda_t)\equiv\begin{cases} 
U_{st}\lambda_{t+\b(s)}&
\mbox{if }s\preceq\a(r(t)), \\ 
0&\mbox{if }s\not\preceq\a(r(t)).
\end{cases}
\end{equation}

\begin{prop} \label{scaf1}
Let $\Phi_1,\ldots,\Phi_n\in K[G]$ and let
$\{\lambda_t:t\in\Z\}$ be a subset of $L$ satisfying
conditions (i) and (ii) in Definition~\ref{scaffold}.
Suppose that for all $t\in\Z$ and $1\le i\le n$ we have
\begin{alignat}{2} \label{hyp1}
v_L(\Phi_i(\lambda_t))&=t+p^{n-i}b_i
&&\text{ if }\a(r(t))_{(n-i)}\ge1, \\
v_L(\Phi_i(\lambda_t))&>t+p^{n-i}b_i
&&\text{ if }\a(r(t))_{(n-i)}=0. \label{hyp2}
\end{alignat}
Then there are $\Psi_1,\ldots,\Psi_n\in K[G]$ such that
$(\{\Psi_i\},\{\lambda_t\})$ is a Galois scaffold for
$L/K$ with precision $\cc=1$.
\end{prop}

\begin{proof}
For $1\le i\le n$ set
$\Psi_i=\Phi_i-\Phi_i(1)\cdot\id_L$.  It follows from
(\ref{hyp1}) and (\ref{hyp2}) that
$v_L(\Phi_i(1))>p^{n-i}b_i$.  Hence for all $t\in\Z$ we
have
\[\Psi_i(\lambda_t)\equiv\Phi_i(\lambda_t)
\pmod{\lambda_{t+p^{n-i}b_i}\M_L},\]
so (\ref{hyp1}) and (\ref{hyp2}) hold with $\Phi_i$
replaced by $\Psi_i$.  If $\a(t)_{(n-i)}\ge1$ then
there is $u_{it}\in\OO_K^{\times}$ such that
\[\Psi_i(\lambda_t)\equiv u_{it}\lambda_{t+p^{n-i}b_i}
\pmod{\lambda_{t+p^{n-i}b_i}\M_L}.\]
Therefore $(\{\Psi_i\},\{\lambda_t\})$ is a Galois
scaffold for $L/K$ with precision 1.
\end{proof}

\section{Stable and semistable extensions}

In this section we outline Bondarko's theory of stable
and semistable extensions \cite{bon1,bon2}.  We also
present some interpretations and refinements of
Bondarko's work which will be useful in the next
section.

     Let $L/K$ be a totally ramified Galois extension of
degree $p^n$ and let $T=\sum_{\sigma\in G}\sigma$ denote
the trace element of $K[G]$.  Define
$\phi:L\otimes_KL\ra L[G]$ by
\[\phi(a\otimes b)=aTb
=\sum_{\sigma\in G}a\sigma(b)\sigma.\]
By Remark~1.2 of \cite{bon1}, $\phi$ is an
isomorphism of vector spaces over $K$.  The following is
proved in Proposition~1.6.1 of \cite{bon1}:

\begin{prop} \label{prod}
Let $\alpha,\beta\in L\otimes_KL$ and write
$\phi(\alpha)=\sum_{\sigma\in G}a_{\sigma}\sigma$ and
$\phi(\beta)=\sum_{\sigma\in G}b_{\sigma}\sigma$.  Then
$\phi(\alpha\beta)=\sum_{\sigma\in G}
a_{\sigma}b_{\sigma}\sigma$.
\end{prop}

     Let $H=\langle(p^n,-p^n)\rangle$ be the subgroup of
$\Z\times\Z$ generated by the element $(p^n,-p^n)$.
Then
\[\F=\{(a,b)\in\Z\times\Z:0\le b<p^n\}\]
is a set of coset representatives for $(\Z\times\Z)/H$.
For $(a,b)\in\Z\times\Z$ write $[a,b]$ for the coset
$(a,b)+H$.  We define a partial order on
$(\Z\times\Z)/H$ by $[a,b]\le[c,d]$ if and only if there
is $(c',d')\in[c,d]$ with $a\le c'$ and $b\le d'$.  We
easily see that $[a,b]\not\le[c,d]$ if and only if
$[c+1,d-p^n+1]\le[a,b]$ (see \cite[(16)]{bon2}).

     Fix a uniformizer $\pi_L$ for $L$ and let $\T$ be
the set of Teichm\"uller representatives for $K$.  Let
$\beta\in L\otimes_KL$.  Then there are uniquely
determined $c_j\in L$ and $a_{ij}\in\T$ such that
\begin{align*}
\beta&=\sum_{j=0}^{p^n-1}c_j\otimes\pi_L^j
=\sum_{(i,j)\in\F}a_{ij}\pi_L^i\otimes\pi_L^j.
\end{align*}
Set
\[R(\beta)=\{[i,j]\in(\Z\times\Z)/H:(i,j)\in\F,
\;a_{ij}\not=0\}.\]

\begin{definition}
Define the diagram of $\beta$ to be
\[D(\beta)=\{[a,b]\in(\Z\times\Z)/H:[i,j]\le[a,b]
\text{ for some }[i,j]\in R(\beta)\}.\]
\end{definition}

     It is observed in Remark~2.4.3 of \cite{bon2} that
while $R(\beta)$ can depend on the choice of uniformizer
$\pi_L$ for $L$, $D(\beta)$ does not depend on this
choice.  Let $G(\beta)$ denote the set of minimal
elements of $D(\beta)$ with respect to $\le$.  Then
$G(\beta)$ is also the set of minimal elements of
$R(\beta)$.

     Let $\delta_{L/K}=\M_L^d$ denote the different
of $L/K$ and set $i_0=d-p^n+1$.  Then $i_0$ is the 0th
index of inseparability of $L/K$
\cite[Prop.\,3.18]{heier}, and $\b(p^n-1)=i_0$
\cite[IV, \S1, Prop.\,4]{cl}.  The following is proved
in Proposition~2.4.2 of \cite{bon2}:

\begin{theorem} \label{shifts}
Let $\xi\in K[G]\smallsetminus\{0\}$, let
$\beta\in L\otimes_KL$ satisfy $\xi=\phi(\beta)$, and
let $a,b\in\Z$.  Then the following are equivalent:
\begin{enumerate}[(a)]
\item $[a,b]\in G(\beta)$.
\item For all $y\in L$ with $v_L(y)=-b-i_0$ we have
$v_L(\xi(y))=a$.
\end{enumerate}
\end{theorem}

     For $\xi\in K[G]$ with $\xi\not=0$ we define
$f_{\xi}:\Z\ra\Z$ by
\[f_{\xi}(a)=\min\{v_L(\xi(y)):y\in\M_L^a\}.\]
Then $f_{\xi}(a+1)\ge f_{\xi}(a)$.  Furthermore, for
every $a\in\Z$ there is $z\in L$ with $v_L(z)=a$ and
$v_L(\xi(z))=f_{\xi}(a)$.

\begin{cor} \label{fxi}
Let $\xi\in K[G]\smallsetminus\{0\}$, let
$\beta\in L\otimes_KL$ satisfy $\xi=\phi(\beta)$, and
let $a,b\in\Z$.  Then the following are equivalent:
\begin{enumerate}[(a)]
\item $[a,b]\in G(\beta)$.
\item $f_{\xi}(-b-i_0)=a$ and $f_{\xi}(-b-i_0+1)>a$.
\end{enumerate}
\end{cor}

\begin{proof}
Suppose $[a,b]\in G(\beta)$.  It follows from
Theorem~\ref{shifts} that for $y\in L$ with
$v_L(y)=-b-i_0$ we have $v_L(\xi(y))=a$.  Therefore
$f_{\xi}(-b-i_0)=a$.  Suppose $f_{\xi}(-b-i_0+1)=a$.
Then there is $z\in L$ with $v_L(z)=-b-i_0+1$ and
$v_L(\xi(z))=a$.  It follows that there is
$u\in\OO_K^{\times}$ such that
\[v_L(\xi(y+uz))=v_L(\xi(y)+u\xi(z))>a.\]
Since $v_L(y+uz)=-b-i_0$ this contradicts
Theorem~\ref{shifts}.  Hence $f_{\xi}(-b-i_0+1)>a$.

     Suppose $f_{\xi}(-b-i_0)=a$ and
$f_{\xi}(-b-i_0+1)>a$.  Then there is $y_0\in L$ such
that $v_L(y_0)=-b-i_0$ and $v_L(\xi(y_0))=a$.  Let
$y\in L$ satisfy $v_L(y)=-b-i_0$.  Then there are
$u\in\OO_K^{\times}$ and $z\in\M_L^{-b-i_0+1}$ such that
$y=uy_0+z$.  It follows that $v_L(\xi(z))>a$, and hence
that $v_L(\xi(y))=v_L(u\xi(y_0)+\xi(z))=a$.  Since this
holds for all $y\in L$ such that $v_L(y)=-b-i_0$ we get
$[a,b]\in G(\beta)$ by Theorem~\ref{shifts}.
\end{proof}

     The following application of Theorem~\ref{shifts}
gives information about the effect that $\xi$ has on
valuations of arbitrary elements of $L^{\times}$.  For a
related result see Proposition~2.5.2 of \cite{bon2}.

\begin{cor} \label{shiftineq}
Let $\xi\in K[G]\smallsetminus\{0\}$, let
$\beta\in L\otimes_KL$ satisfy $\xi=\phi(\beta)$, and
let $a,b\in\Z$.  Then the following are equivalent:
\begin{enumerate}[(a)]
\item $[a,b]\in D(\beta)$.
\item $f_{\xi}(-b-i_0)\le a$.
\end{enumerate}
\end{cor}

\begin{proof}
Suppose $[a,b]\in D(\beta)$.  Then there is
$[c,d]\in G(\beta)$ such that $[c,d]\le[a,b]$.  We may
assume that $c\le a$ and $d\le b$.  It follows from
Corollary~\ref{fxi} that
\[f_{\xi}(-b-i_0)\le f_{\xi}(-d-i_0)=c\le a.\]
Suppose $[a,b]\not\in D(\beta)$.  Let $d$ be the
largest integer such that $d\le b$ and there exists $c$
such that $[c,d]\in G(\beta)$.  By Corollary~\ref{fxi}
we have
\[f_{\xi}(-b-i_0)\le f_{\xi}(-d-i_0)=c.\]
Suppose $f_{\xi}(-b-i_0)<c$.  Then there is $d'$ such
that $d<d'\le b$ and
\[f_{\xi}(-d'-i_0)<f_{\xi}(-d'-i_0+1).\]
Setting $c'=f_{\xi}(-d'-i_0)$ we get
$[c',d']\in G(\beta)$ by Corollary~\ref{fxi}.  This
contradicts the maximality of $d$, so we must have
$f_{\xi}(-b-i_0)=c$.  Furthermore, since
$[c,d]\not\le[a,b]$ we have $a<c$.  Hence
$f_{\xi}(-b-i_0)>a$.
\end{proof}
 
     For $\beta\in L\otimes_KL$ with $\beta\not=0$ set
\[d(\beta)=\min\{i+j:[i,j]\in D(\beta)\}.\]
Define the diagonal of $\beta$ to be
\[N(\beta)=\{[i,j]\in D(\beta):i+j=d(\beta)\}.\]
Then $N(\beta)\subset G(\beta)$.

\begin{definition}
Let $L/K$ be a totally ramified Galois extension of
degree $p^n$.  Say that $L/K$ is a semistable extension
if there is $\beta\in L\otimes_KL$ such that
$\phi(\beta)\in K[G]$, $p\nmid d(\beta)$, and
$|N(\beta)|=2$.  Say that $L/K$ is a stable extension if
$\beta$ can be chosen to satisfy the additional
condition $G(\beta)=N(\beta)$.
\end{definition}

     In \cite[Th.\,4.3.2]{bon2} sufficient conditions
are given for an ideal $\M_L^h$ in a semistable
extension $L/K$ to be free over its associated order.
In some cases where $L/K$ is stable these conditions are
shown to be necessary as well.

     The following congruences imply that for $n\ge2$
most totally ramified Galois extensions of degree $p^n$
are not semistable.

\begin{prop} \label{breaks}
Let $L/K$ be a semistable extension and let
$b_1\le b_2\le\cdots\le b_n$ be the lower ramification
breaks of $L/K$, counted with multiplicity.  Then
$b_i\equiv-i_0\pmod{p^n}$ for $1\le i\le n$.
\end{prop}

\begin{proof}
This follows from Propositions~3.2.2 and 4.3.2.1 in
\cite{bon2}.
\end{proof}

     Motivated by the theory of scaffolds, we give a
definition of precision for semistable extensions.

\begin{definition}
Let $L/K$ be a totally ramified Galois extension of
degree $p^n$ and let $\cc\ge1$.  We say that $L/K$ is
semistable with precision $\cc$ if there is
$\beta\in L\otimes_KL$ such that $\phi(\beta)\in K[G]$,
$p\nmid d(\beta)$, $|N(\beta)|=2$, and
$a+b\ge d(\beta)+\cc$ for all
$[a,b]\in G(\beta)\smallsetminus N(\beta)$.
\end{definition}

     If $L/K$ is stable then we may choose $\beta$ so
that $G(\beta)=N(\beta)$.  Hence $L/K$ is semistable
with infinite precision in this case.  On the other
hand, if $L/K$ is semistable with sufficiently high
precision then $L/K$ is stable:

\begin{prop} \label{stable}
Let $L/K$ be a totally ramified Galois extension of
degree $p^n$, and let $h\in\SS_{p^n}$ satisfy
$h\equiv i_0\pmod{p^n}$.  Suppose $L/K$ is semistable
with precision
\[\cc\ge\max\{h-1,p^n-h-1\}.\]
Then $L/K$ is a stable.
\end{prop}

\begin{proof}
Since $L/K$ is semistable with precision $\cc$ there is
$\alpha\in L\otimes_KL$ with $\phi(\alpha)\in K[G]$,
$p\nmid d(\alpha)$, $|N(\alpha)|=2$, and
$e+f\ge d(\alpha)+\cc$ for all
$[e,f]\in G(\alpha)\smallsetminus N(\alpha)$.  The proof
of Proposition~3.2.1 in \cite{bon2} shows that
$N(\alpha)=\{[d(\alpha),0],[0,d(\alpha)]\}$.  By
Proposition~4.3.2.1 in \cite{bon2} we have
$d(\alpha)\equiv i_0\equiv h\pmod{p^n}$.  Set
$m=(d(\alpha)-h)/p^n$.  Then $m\in\Z$, so we may define
$\beta=\pi_K^{-m}\alpha$.  We get
$\phi(\beta)=\pi_K^{-m}\phi(\alpha)\in K[G]$ and
$D(\beta)=[-mp^n,0]+D(\alpha)$.  It follows that
$N(\beta)=\{[h,0],[0,h]\}$ and $e+f\ge h+\cc$ for all
$[e,f]\in G(\beta)\smallsetminus N(\beta)$.

     Suppose there is $(e,f)\in\F$ such that
$[e,f]\in G(\beta)\smallsetminus N(\beta)$.  Then
$[h,0]\not\le[e,f]$ and $[0,h]\not\le[e,f]$.  It follows
that
\begin{align} \label{ineq1}
[e+1,f-p^n+1]&\le[h,0] \\
[e+1,f-p^n+1]&\le[0,h]. \label{ineq2}
\end{align}
By (\ref{ineq1}) we have $e+1\le h$.  Hence if
$0\le f\le h-1$ then
\[e+f\le2h-2<2h-1.\]
If $h\le f\le p^n-1$ then by (\ref{ineq2}) we have
$e+1\le0$, and hence
\[e+f\le-1+p^n-1<p^n-1.\]
In either case we get $e+f<h+\cc$, a contradiction.
Hence $G(\beta)\smallsetminus N(\beta)=\varnothing$, so
$L/K$ is stable.
\end{proof}

\section{Semistable extensions and Galois scaffolds}
\label{semscaf}

Let $L/K$ be a totally ramified extension of degree
$p^n$.  In this section we show that $L/K$ is semistable
if and only if $L/K$ has a Galois scaffold.  We also
show that if $L/K$ has a Galois scaffold with
sufficiently high precision then $L/K$ is stable.

\begin{theorem} \label{semi}
Let $L/K$ be a totally ramified Galois extension of
degree $p^n$ which is semistable.  Then $L/K$ admits a
Galois scaffold with precision $\cc=1$.
\end{theorem}

\begin{proof}
Let $h\in\SS_{p^n}$ satisfy $h\equiv i_0\pmod{p^n}$.
Then by Propositions~3.2.1 and 4.3.2.1 of \cite{bon2}
there is $\beta\in L\otimes_KL$ such that
$\phi(\beta)\in K[G]$ and $N(\beta)=\{[0,h],[h,0]\}$.
For $1\le i\le n$ set
$\Theta_i=\phi(\beta^{p^n-p^{n-i}-1})$.  Since
$\phi(\beta)\in K[G]$ it follows from
Proposition~\ref{prod} that $\Theta_i\in K[G]$.  We
clearly have $d(\beta^{p^n-p^{n-i}-1})
=(p^n-p^{n-i}-1)h$ and
\begin{align} \label{N}
N(\beta^{p^n-p^{n-i}-1})
&=\left\{[sh,(p^n-p^{n-i}-1-s)h]:s\in\SS_{p^n},\;
p\nmid\binom{p^n-p^{n-i}-1}{s}\right\}.
\end{align}
Let $\{\lambda_t:t\in\Z\}$ be a set of elements of $L$
such that $v_L(\lambda_t)=t$ and
$\lambda_{t'}\lambda_t^{-1}\in K$ for all $t,t'\in\Z$
with $t'\equiv t\pmod{p^n}$.  It follows from (\ref{N})
and Corollary~\ref{shiftineq} that for $t\in\Z$ we have
\begin{equation} \label{Theta}
v_L(\Theta_i(\lambda_t))\ge t+(p^n-p^{n-i}-1)h+i_0.
\end{equation}
Since $p\nmid h$ there is a unique $s_i(t)\in\SS_{p^n}$
such that
\begin{equation} \label{cond2}
t\equiv-(p^n-p^{n-i}-1-s_i(t))h-i_0\pmod{p^n}.
\end{equation}
By applying Theorem~\ref{shifts} and
Corollary~\ref{shiftineq} to (\ref{N}) we see that
equality holds in (\ref{Theta}) if and only if
$p\nmid\binom{p^n-p^{n-i}-1}{s_i(t)}$.
It follows from Lucas's theorem that this is
equivalent to $s_i(t)_{(n-i)}\not=p-1$, which holds if
and only if $p^{n-i}+s_i(t)<p^n$ and
$(p^{n-i}+s_i(t))_{(n-i)}\ge1$.  

     Since $h\equiv i_0\pmod{p^n}$ it follows from
(\ref{cond2}) that
\[t\equiv(p^{n-i}+s_i(t))h\pmod{p^n}.\]
By Proposition~\ref{breaks} we have
$h\equiv-b_j\pmod{p^n}$ for $1\le j\le n$.  Therefore
\begin{equation} \label{art}
\a(r(t))\equiv p^{n-i}+s_i(t)\pmod{p^n}.
\end{equation}
If we have equality in (\ref{Theta}) then by the
preceding paragraph we get $\a(r(t))=p^{n-i}+s_i(t)$
and $\a(r(t))_{(n-i)}\ge1$.
Conversely, if $\a(r(t))_{(n-i)}\ge1$ then
$\a(r(t))\ge p^{n-i}$, and hence
$p^{n-i}+s_i(t)=\a(r(t))<p^n$ by (\ref{art}).  It
follows from the preceding paragraph that equality holds
in (\ref{Theta}) in this case.  Thus equality holds in
(\ref{Theta}) if and only if $\a(r(t))_{(n-i)}\ge1$.

     For $1\le i\le n$ set
\[v_i=\frac{(p^n-p^{n-i}-1)h+i_0-p^{n-i}b_i}{p^n}.\]
Then $v_i\in\Z$, so we may define
$\Phi_i=\pi_K^{-v_i}\Theta_i\in K[G]$.  Using
(\ref{Theta}) we get
$v_L(\Phi_i(\lambda_t))\ge t+p^{n-i}b_i$, with equality
if and only if $\a(r(t))_{(n-i)}\ge1$.  Hence by
Proposition~\ref{scaf1} there are $\Psi_i\in K[G]$ such
that $(\{\Psi_i\},\{\lambda_t\})$ is a Galois scaffold
for $L/K$ with precision 1.
\end{proof}

\begin{theorem} \label{prec}
Let $L/K$ be a totally ramified Galois extension of
degree $p^n$ which has a Galois scaffold
$(\{\Phi_i\},\{\lambda_t\})$ with precision $\cc\ge1$.
Then $L/K$ is semistable with precision $\cc$.
\end{theorem}

\begin{proof}
Set $\xi=\Psi^{(p^n-2)}$.  It follows from (\ref{Psis})
that for every $t\in\Z$ there is
$U_{p^n-2,t}\in\OO_K^{\times}$ such that the following
holds modulo $\lambda_{t+\b(p^n-2)}\M_L^{\cc}$:
\begin{equation} \label{Psipn2}
\xi(\lambda_t)\equiv\begin{cases} 
U_{p^n-2,t}\lambda_{t+\b(p^n-2)}&
\mbox{if }p^n-2\preceq\a(r(t)), \\ 
0&\mbox{if }p^n-2\not\preceq\a(r(t)).
\end{cases}
\end{equation}
Therefore we have $v_L(\xi(\lambda_t))\ge t+\b(p^n-2)$,
with equality if and only if $p^n-2\preceq\a(r(t))$.
Now let $y\in L^{\times}$ be arbitrary and set
$v_L(y)=t$.  It follows from the above that
$v_L(\xi(y))\ge t+\b(p^n-2)$, again with equality if and
only if $p^n-2\preceq\a(r(t))$.  The condition
$p^n-2\preceq\a(r(t))$ is equivalent to
$\a(r(t))\in\{p^n-1,p^n-2\}$, which is valid if and only
if we one of the following congruences holds:
\begin{alignat}{2} \label{cong1}
t&\equiv-\b(p^n-1)&&\pmod{p^n} \\
t&\equiv-\b(p^n-2)&&\pmod{p^n}. \label{cong2}
\end{alignat}
It follows that $f_{\xi}(t)\ge t+\b(p^n-2)$, with
equality if and only if either (\ref{cong1}) or
(\ref{cong2}) holds.  Let $\beta\in L\otimes_KL$ be
such that $\phi(\beta)=\xi$.  By
Corollary~\ref{shiftineq} we have
\[[t+\b(p^n-2)-1,-i_0-t]\not\in D(\beta)\]
for all $t\in\Z$.  Furthermore,
$[t+\b(p^n-2),-i_0-t]\in D(\beta)$ if and only if
either (\ref{cong1}) or (\ref{cong2}) holds.  It follows
that
\[d(\beta)=\b(p^n-2)-i_0=-b_n\]
and $N(\beta)=\{[-b_n,0],[0,-b_n]\}$.  Since
$p\nmid b_n$ this shows that $L/K$ is semistable.

     Suppose $[e,f]\in G(\beta)$ and
$[e,f]\not\in N(\beta)$.  Then
$f\not\equiv0,-b_n\pmod{p^n}$, and hence
\begin{alignat*}{2}
-f-i_0&\not\equiv-\b(p^n-1)&&\pmod{p^n} \\
-f-i_0&\not\equiv-\b(p^n-2)&&\pmod{p^n}.
\end{alignat*}
It follows from Theorem~\ref{shifts} and (\ref{Psipn2})
that
\begin{align*}
e&=v_L(\xi(\lambda_{-f-i_0})) \\
&\ge-f-i_0+\b(p^n-2)+\cc \\
&=-f-b_n+\cc.
\end{align*}
Hence $e+f\ge-b_n+\cc=d(\beta)+\cc$, so $L/K$ is
semistable with precision $\cc$.
\end{proof}

\begin{cor} \label{stab}
Let $L/K$ be a totally ramified Galois extension of
degree $p^n$ and let $h\in\SS_{p^n}$ satisfy
$h\equiv i_0\pmod{p^n}$.  Suppose that $L/K$ has a
Galois scaffold $(\{\Phi_i\},\{\lambda_t\})$ with
precision
\[\cc\ge\max\{h-1,p^n-h-1\}.\]
Then $L/K$ is a stable extension.
\end{cor}

\begin{proof}
This follows from Theorem~\ref{prec} and
Proposition~\ref{stable}.
\end{proof}

\begin{remark}
The proofs of Theorem~\ref{prec} and
Corollary~\ref{stab} don't require part (iii) of
Definition~\ref{scaffold}, which states that $\Psi_i$
lies in the augmentation ideal of $K[G]$.
\end{remark}

\begin{remark}
It would be interesting to get some sort of converse to
Theorem~\ref{prec} or Corollary~\ref{stab}.
Unfortunately, when $L/K$ is semistable with precision
$\cc$ the Galois scaffold for $L/K$ produced by
Theorem~\ref{semi} does not seem to inherit this
precision.  Constructing a Galois scaffold with
precision $\cc>1$ for a stable or semistable extension
$L/K$ would require at least a refinement of
Theorem~\ref{semi}, or perhaps a completely new approach.
\end{remark}

\begin{cor}
Let $L/K$ be a totally ramified Galois extension of
degree $p^n$ and let $G=\Gal(L/K)$.  Then the following
statements are equivalent:
\begin{enumerate}[(a)]
\item $L/K$ admits a Galois scaffold with precision
$\cc$ for some $\cc\ge1$.
\item $L/K$ is semistable.
\item Let $\rho\in L$ satisfy
$v_L(\rho)\equiv-i_0\pmod{p^n}$.  Then for all
$\xi\in K[G]$ and $\lambda\in L^{\times}$ we have
\[v_L(\xi(\rho))-v_L(\rho)\le
v_L(\xi(\lambda))-v_L(\lambda).\]
\end{enumerate}
\end{cor}

\begin{proof}
The equivalence of (a) and (b) follows from
Theorems~\ref{semi} and \ref{prec}.  The equivalence of
(b) and (c) is proved in Theorem~4.4 of \cite{bon2}.
\end{proof}

     We conclude with some applications of our main
results.

\begin{cor}
Let $L/K$ be a totally ramified Galois extension of
degree $p^n$ which admits a Galois scaffold with
precision $\cc$ for some $\cc\ge1$.  Then the lower
ramification breaks of $L/K$ satisfy
$b_i\equiv-i_0\pmod{p^n}$ for $1\le i\le n$.
\end{cor}

\begin{proof}
It follows from Theorem~\ref{prec} that $L/K$ is a
semistable extension.  Hence by Proposition~\ref{breaks}
we have $b_i\equiv-i_0\pmod{p^n}$ for all $i$.
\end{proof}

     Let $K$ be a finite unramified extension of $\Q_p$,
with $K\not=\Q_p$.  In \cite[\S5.1]{bon2} a totally
ramified elementary abelian $p$-extension $L/K$ is
constructed which is semistable, but not stable.  The
following shows that there are no examples of elementary
abelian $p$-extensions in characteristic $p$ with these
properties.

\begin{prop}
Let $\ch(K)=p$ and let $L/K$ be a totally ramified
elementary abelian $p$-extension.  If $L/K$ is
semistable then it is stable.
\end{prop}

\begin{proof}
It follows from Theorem~\ref{semi} that $L/K$ has a
Galois scaffold $(\{\Psi_i\},\{\lambda_t\})$ with
precision 1.  Since $\ch(K)=p$ we have $\Psi_i^p=0$ for
$1\le i\le n$.  Hence by Theorem~A.1(ii) of \cite{bce}
the extension $L/K$ has a Galois scaffold with precision
$\infty$.  Therefore by Corollary~\ref{stab} $L/K$ is
stable.
\end{proof}

\end{document}